\documentclass[12pt]{amsart}
\usepackage[a4paper,margin=1in]{geometry}
\usepackage{amsmath,amsthm,amscd,amsfonts,amssymb,amsbsy,enumerate}
\usepackage{setspace}
\usepackage{booktabs}

\usepackage{mathrsfs} 
\usepackage{hyperref}
\usepackage{url}
\usepackage{wrapfig}
\usepackage{caption}
\usepackage{subcaption}
\usepackage{mathtools}
\usepackage{xcolor}
\usepackage{graphicx}

\newtheorem{theorem}{Theorem}[section]
\newtheorem{definition}[theorem]{Definition}
\newtheorem{remark}[theorem]{Remark}
\newtheorem{corollary}[theorem]{Corollary}

\newcommand{\N}{\mathbb{N}}
\newcommand{\R}{\mathbb{R}}

\newcommand{\cO}{\mathcal{O}}
\newcommand{\cS}{\mathcal{S}}

\newcommand{\loc}{\text{loc}}

\newcommand{\be}{\begin{equation}}
\newcommand{\ee}{\end{equation}}

\newcommand{\abs}[1]{{\left| #1 \right|}}

\newcommand{\inn}[2]{{\langle #1,#2 \rangle}}

\DeclareMathOperator{\supp}{supp}

\begin{document}

\title{Fourier Representations of Fractional B-Splines via Generalized Stirling-type Polynomials}

\author{Damla G\"un}
\address{Department of Mathematics, Akdeniz University, Antalya, T\"urkiye}
\email{damlagun@akdeniz.edu.tr}
\author{Peter Massopust}
\address{Department of Mathematics, Technical University of Munich, Garching b. Munich, Germany}
\email{massopust@ma.tum.de}
\author{Yilmaz Simsek}
\address{Department of Mathematics, Akdeniz University, Antalya, T\"urkiye}
\email{ysimsek@akdeniz.edu.tr}


\date{}
	
\begin{abstract}
In this paper, using generating functions methods, we investigate fractional B-splines and their connections with Fourier analysis, and establish connections with generalized Array polynomials in terms of the Stirling-type numbers and distribution theory. Employing a generating-function approach inspired by recent results of Simsek \cite{Simsek2013}, we derive some novel formulas and relations involving Fourier-type expansion for fractional B-splines with the aid of these polynomials and numbers. Other main contribution of this paper is to provide a new type Fourier expansion of fractional B-splines by using these numbers and polynomials. This  expansion show that fractional B-splines as infinite linear combinations of derivatives of the Dirac delta on the distributional sense. Furthermore, we establish an explicit shifted distributional representation, which characterize the action of fractional B-splines with test functions. Finally, by using the Mittag–Leffler function, we construct new generating function for a new class of fractional spline polynomials. The results of this paper provide many applications with connects spline theory, fractional calculus, combinatorial structures and special functions. 
\end{abstract}		

\maketitle

\section{Introduction}
Spline functions play a central role in approximation theory, numerical analysis, and signal processing. In particular, polynomial B-splines, introduced by Schoenberg, provide a powerful and flexible basis for representing piecewise polynomial functions with prescribed smoothness. These functions possess remarkable properties such as compact support, stability, and efficient recursive constructions, which make them fundamental tools in both theory and applications.

A key feature of B-splines is their Fourier representation, which reveals a deep connection between spline theory and harmonic analysis. In recent years, this classical framework has been extended to fractional and complex orders, leading to the notion of complex B-splines. These generalized splines form a continuous family indexed by a complex parameter, allowing a refined control of smoothness through the real part and phase information through the imaginary part. Moreover, complex B-splines can be characterized as solutions of fractional differential equations with discrete Dirac-type sources, linking spline theory with fractional calculus and distribution theory.

As for the theory of generating functions have many important applications in various applied sciences. Especially, generating functions for spline theory has been rarely studied by researchers. The motivation of this paper is to use generating function method with their functional equations to derive many novel relations, formulas and series representations. Our results are related to fractional B-spline, Lizorkin space, fractional operators, Array polynomials, Stirling-type polynomials, combinatorial structures etc. Therefore, we establish novel connections between fractional B-splines, Fourier analysis, and Array polynomials and generalized Stirling-type numbers. We also give not only Fourier-type expansion of fractional B-splines in terms of generalized combinatorial coefficients, but also fractional B-splines as infinite linear combinations of derivatives of the Dirac delta on distributions. As an additional contribution, we construct a new class of fractional spline polynomials and establish their generating function, revealing a connection with Mittag--Leffler-type structures.
\section{Brief Review of Spline Theory}
In this section, we briefly recall the basic notions and some results on spline functions that will be used throughout the paper.

\begin{definition}\label{Def2.1}
Let $X = \{x_0 < x_1 < \cdots < x_{k+1}\} \subset \mathbb{R}$ be a knot set. 
A polynomial spline of order $n$ is a function $s:[x_0,x_{k+1}]\to\mathbb{R}$ such that
\begin{enumerate}
\item[(i)] $s$ is a polynomial of degree at most $n-1$ on each interval $[x_{i-1},x_i]$ and
\item[(ii)] $s \in C^{n-2}[x_0,x_{k+1}]$.
\end{enumerate}
If the knot set is uniform, i.e., $X \subset \mathbb{Z}$, then $s$ is called a {cardinal polynomial spline}.
\end{definition}
For $n=1$, $C^{-1}[x_0,x_{k+1}]$ is defined as the space of piecewise constant functions on $[x_0, x_{k+1}]$. Note that condition (ii) in Definition \ref{Def2.1} is equivalent to requiring that $s\vert_{(x_{i-1},x_i)}$ be a solution of the linear diﬀerential equation $D^n y = 0$, where $D$ denotes the ordinary diﬀerential operator on functions.

For a fixed knot set $X$, polynomial splines of order $n$ form a real vector space of dimension $n+k$. A convenient basis for this vector space is given by Schoenberg's polynomial B-splines.
\begin{definition}
Let $\chi$ denote the characteristic or indicator function of the unit interval $[0,1]$. For $1<n\in\N$, set
\begin{align}
&B_1 :[0,1] \to [0,1], \quad x\mapsto \chi(x), \nonumber\\
&B_n (x) := (B_{n-1} * B_1)(x) = \int_0^1 B_{n-1}(x-t)dt.
\end{align}
An element of the family $\{B_n : n\in\N\}$ is called a cardinal polynomial B-spline of order $n$.
\end{definition}
\begin{remark}
As we deal exclusively with cardinal splines in this paper, we will drop the adjective ``cardinal" from now on.
\end{remark}

Polynomial B-splines admit an explicit representation of the form
\begin{equation}
B_n(x)=\frac{1}{\Gamma(n)}\sum_{k=0}^{n}(-1)^k\binom{n}{k}(x-k)_+^{n-1},
\label{Bspline}
\end{equation}
where
\begin{equation}
\binom{n}{k}=	\frac{\Gamma(n+1)}
{\Gamma(k+1)\Gamma(n-k+1)} \label{gamma}
\end{equation} 
and
\[
x_+^n :=\begin{cases}
x^n, & x>0;\\
0, & x\le 0.
\end{cases}
\]
It can be shown that every spline $s:[x_0,x_{k+1}]\to\mathbb{R}$ of order $n$ is a linear combination of shifted B-splines:
\[
s(x) = \sum_{j=-n+1}^k c_j B_n(x-j),
\]
for certain $c_j\in\R$. This result implies that the investigation of properties of splines can be reduced to the investigation of these properties for the basis B-splines. The graphs of some B-splines are shown in Figure \ref{Fig1}.
\begin{figure}[h!]
    \centering
    \includegraphics[width=8cm,height= 3cm]{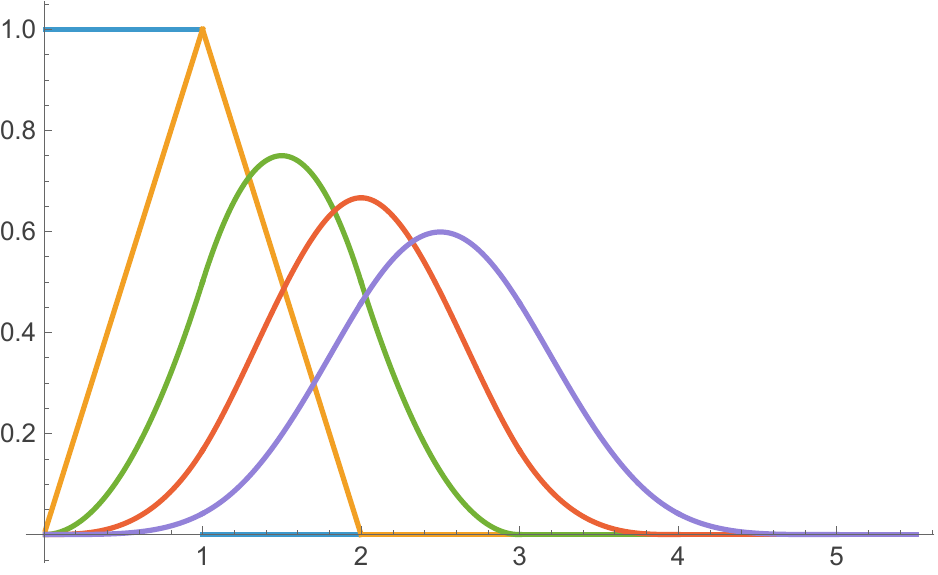}
    \caption{B-splines of orders 1 to 5.}
    \label{Fig1}
\end{figure}

Polynomial B-splines enjoy, among others, the following properties:
\begin{enumerate}
\item $\forall\,n\in\N$: $B_n\in C^{n-2}$ and $\supp B_n = [0,n]$;
\item Recursion relation: $\forall\,1<n\in\N\;\forall\,x\in\R$: $B_n(x) = \frac{x}{n-1} B_{n-1}(x) + \frac{n-x}{n-1} B_{n-1}(x-1)$;
\item Given $f \in C^n[a,b]$, the error of approximating $f$ by polynomial B-splines of order $n$ on a uniform grid of mesh size $h$ is $\mathcal{O}(h^n)$.
\end{enumerate}
For more details and further properties of B-splines, we refer to \cite{deBoor1978} of \cite{Massopust2012}.

However, Schoenberg’s polynomial B-splines have two drawbacks:
\begin{enumerate}
\item $\{B_n : n\in\N\}$ forms a \emph{discrete} family of functions of increasing smoothness, but there exist functions which belong to spaces of non-integral smoothness, for instance, the H\"older spaces $C^{n+\alpha}$, $n\in\N_0$, $0 < \alpha \leq 1$. Such functions cannot be well approximated by B-splines of integral order $n$ since the approximation error depends only on $n$ and not on $n + \alpha$.
\item The family $\{B_n\}$ interpolates/approximates primarily point values but cannot
be used to obtain phase or directional information.
\end{enumerate}
In order to overcome these drawbacks, more general B-splines were recently defined. 
Fractional B-splines, i.e., splines of real order were constructed in \cite{UnserBlu2000} to produce a continuous family of functions, with respect to smoothness, for H\"older spaces, and B-splines of complex order (cf. \cite{ForsterUnserBlu2006}), which also included phase information through the imaginary part of the complex order.

The construction of fractional and complex B-splines is based on the Fourier transform of $B_n$ which is given by
\begin{equation}
\widehat{B_n}(\omega) := \mathcal{F}\{B_n\}(\omega) := \int_\R B_n(x) e^{- i x \omega} dx
=\left(\frac{1-e^{-i\omega}}{i\omega}\right)^n.
\label{FourierBn}
\end{equation}
Since in this paper our focus lies on fractional B-splines, we present only their definition and properties. The interested reader is referred to \cite{UnserBlu2000} for more details and proofs.
\begin{definition}
Let $\alpha > 1$. The fractional B-spline of order $\alpha$ is defined in the Fourier domain by
\begin{equation}
\widehat{B_\alpha}(\omega) := \left(\frac{1-e^{-i\omega}}{i\omega}\right)^\alpha.
\label{FourierFractional}
\end{equation}
\end{definition}
In the time domain, fractional B-splines admit the representation
\begin{equation}
B_\alpha(x)=\frac{1}{\Gamma(\alpha)}\sum_{k=0}^{\infty}(-1)^k\binom{\alpha}{k}(x-k)_+^{\alpha-1}.
\label{TimeFractional}
\end{equation}
For the sake of completeness, we list some properties of fractional B-splines (cf.\cite{UnserBlu2000}).
\begin{enumerate}
    \item For $\alpha > 1$, $B_\alpha\in L^1(\R)\cap L^2(\R)$.
    \item $B_\alpha = \cO (\abs{x}^{-m})$ for $m < \alpha + 1$ as $\abs{x}\to\infty$.
    \item $B_\alpha$ reproduces polynomials up to order $\lceil\alpha\rceil$.
    \item For $\alpha > 1$, $B_\alpha$ is $(\alpha-1)$--H\"older continuous.
\end{enumerate}

In Figure \ref{Fig2}, the graphs of some fractional B-splines are depicted.
\begin{figure}[h!]
    \centering
    \includegraphics[width=8cm, height=3.5cm]{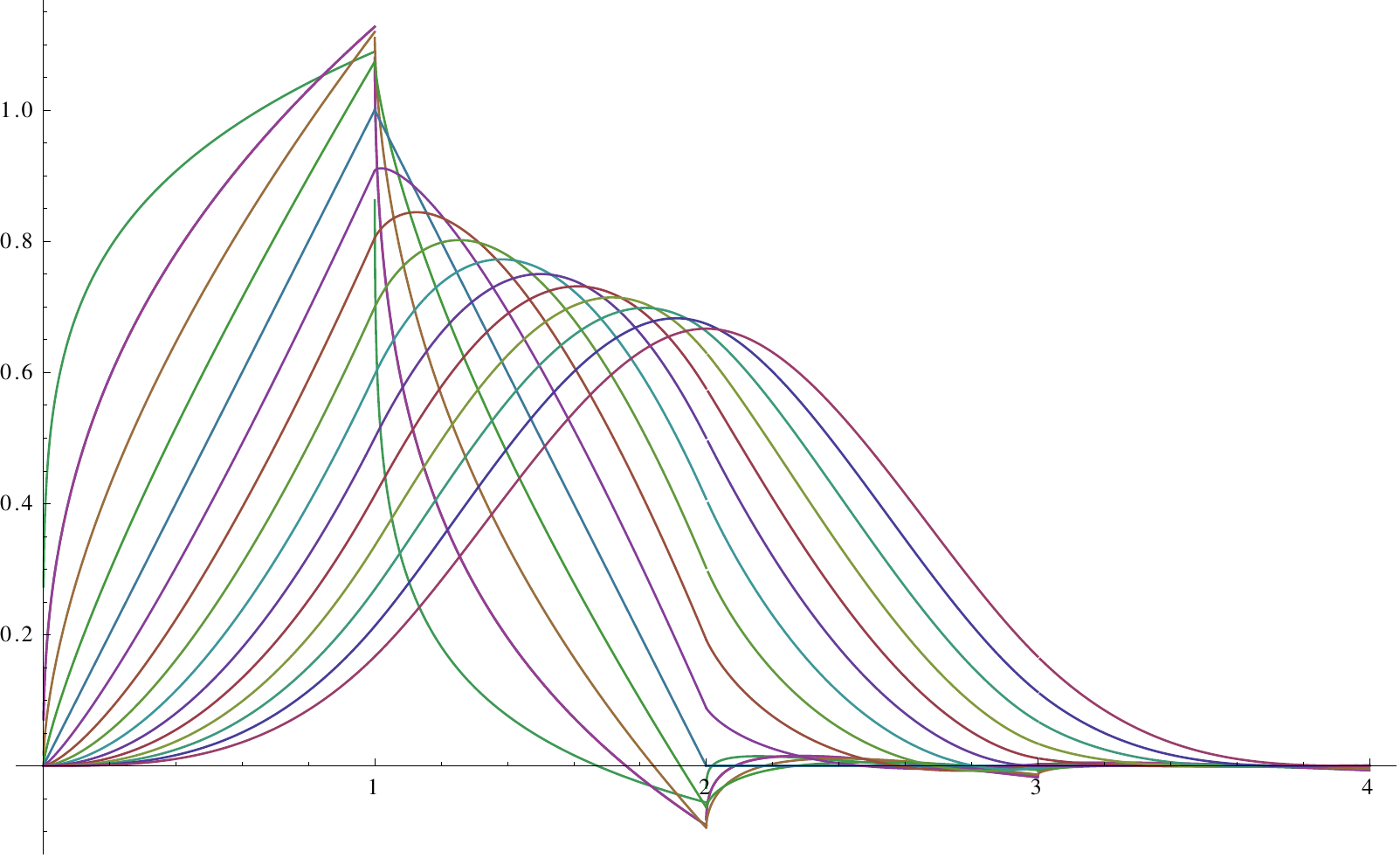}
    \caption{A family of fractional B-splines for $\alpha = 1 + m\cdot 0.25$, $m = 1, \ldots, 12$.}
    \label{Fig2}
\end{figure}

The functions $B_\alpha$ satisfy fractional distributional differential equations of the form
\begin{equation}
D^\alpha f=\sum_{k=0}^{\infty} a_k \delta(x-k),
\end{equation}
which motivates the definition of splines of fractional order (cf. \cite{Massopust2022}). A brief review of fractional operators in the subject of the next section.


%
\section{Some Results About Fractional Operators}
In this short section, we introduce those fractional operators that are relevant for the remainder of this paper. To this end, denote by $\mathcal{S}(\mathbb{R})$ the Schwartz space of rapidly decreasing functions on $\mathbb{R}$:
\[
\mathcal{S}(\mathbb{R}):=\left\{
f \in C^\infty(\mathbb{R}) :
\forall m,n \in \mathbb{N}_0,\;
\sup_{x \in \mathbb{R}} |x^m (D^n f)(x)| < \infty
\right\}.
\]
Recall that $\cS(\R) \subset L^p(\R)$, for all $1\leq p\leq \infty$. The Lizorkin space $\Psi$ is defined by
\[
\Psi := \Psi(\R) :=\{\psi\in\mathcal{S}(\mathbb{R}): D^m\psi(0)=0,\ \forall m\in\N_0\}
\]
and its restriction to the nonnegative real axis by
\[
\Psi_+ := \{f\in \Psi : \supp f \subseteq [0,\infty)\}.
\]
For $\alpha > 0$, define a kernel function by
\be\label{Kalpha}
K_\alpha : \mathbb{R} \to \mathbb{\R}, \qquad
K_\alpha(x) := \frac{x_+^{\,\alpha-1}}{\Gamma(\alpha)}.
\ee
To proceed, we require some definitions and results from the theory of distributions and fractional operators. For a review of distribution theory, we refer to \cite{GelfandShilov1959,Schwartz1966, Zemanian1987} and for fractional derivative and integral operators to \cite{Podlubny1999,Samko1993}.

We note that for $\alpha >0$, $K_\alpha\in L^1_{\loc}$, and we can regard it as a regular distribution in the space $\Psi_+'$.

\begin{definition}[\textit{cf}. \cite{Massopust2022,Samko1993}]
The fractional derivative operator $D^\alpha$ on $\Psi_+'$ is defined by
\[
\inn{D^\alpha f}{\varphi} := \inn{D^n(f * K_{n-\alpha})}{\varphi}, \quad n=\lceil \alpha\rceil,
\]
and the fractional integral operator $D^{-\alpha}$ on $\Psi_+'$ by
\[
\inn{D^{-\alpha}f}{\varphi} = \inn{f * K_\alpha}{\varphi}, 
\]
where $\inn{\cdot}{\cdot}$ denotes the pairing between distributions and test functions, and $*$ the convolution of distributions.
\end{definition}


A key identity is
\begin{equation}
D^\alpha\left[\frac{(x-k)_+^{\alpha-1}}{\Gamma(\alpha)}\right]=\delta(x-k),
\label{D}
\end{equation}
which plays a fundamental role in the construction of fractional splines. Here, $\delta$ denotes the Dirac delta distribution defined by
\[
\langle \delta(x-a),\varphi\rangle=\varphi(a),\quad a\in \R, \quad \varphi\in\Psi_+.
\]
\section{Fractional Operators and Distributional Framework}

In this section, we briefly recall the fractional operators and distributional setting used in the sequel. Our analysis is carried out within the framework of tempered distributions and Lizorkin spaces, which provide a natural setting for fractional differentiation.

Let $\mathcal{S}(\mathbb{R})$ denote the Schwartz space of rapidly decreasing smooth functions, and let $\Psi \subset \mathcal{S}(\mathbb{R})$ be the Lizorkin space defined by
\[
\Psi := \{ \psi \in \mathcal{S}(\mathbb{R}) : D^m \psi(0) = 0, \ \forall m \in \mathbb{N}_0 \}.
\]
We also define $\Psi_+ := \{ \psi \in \Psi : \operatorname{supp}(\psi) \subset [0,\infty) \}$.

For $\alpha > 0$, recall the definition of the kernel $K_\alpha$ in \eqref{Kalpha}. It is needed for the following definitions.

The fractional derivative $D^\alpha$ and fractional integral $D^{-\alpha}$ are defined on the space of distribution $\Psi_+'$ via convolution:
\[
\langle D^\alpha f, \varphi \rangle := \langle D^n (f * K_{n-\alpha}), \varphi \rangle, \quad n = \lceil \alpha \rceil,
\]
\[
\langle D^{-\alpha} f, \varphi \rangle := \langle f * K_\alpha, \varphi \rangle.
\]

A fundamental identity, which plays a central role in the construction of fractional B-splines, is given by
\begin{equation}
D^\alpha\left[\frac{(x-k)_+^{\alpha-1}}{\Gamma(\alpha)}\right]=\delta(x-k),
\label{D}
\end{equation}
where $ \alpha > 1$ and $\delta$ denotes the Dirac delta distribution.

This identity allows fractional B-splines to be interpreted as solutions of fractional distributional differential equations with discrete sources.
\section{Series representation and Generating function for a new class of fractional spline polynomials}
In this section, we construct generating function for a new class of fractional spline polynomials, which involves the function $\widehat{B_{n}}(\omega )$. We give some fundamental properties of this function and this new class of fractional spline polynomials.

It is time to give generating function for the function $\widehat{B_{n}}(\omega )$ by the following theorem:
\begin{theorem} Let $\omega$ be a real parameter.  Generating function for the function $\widehat{B_{n}}(\omega )$ is given by
\begin{equation}
\sum_{n=0}^{\infty }\widehat{B_{n}}(\omega )\frac{t^{n}}{n!}=e^{\frac{%
1-e^{-i\omega }}{i\omega }t}.  \label{gp-1}
\end{equation}
\end{theorem}
\begin{proof}
Using \eqref{FourierBn}, we obtain
\[
\sum_{n=0}^{\infty}
\widehat{B_{n}}(\omega)\frac{t^n}{n!}=
\sum_{n=0}^{\infty}
\left(\frac{1 - e^{-i\omega}}{i\omega}
\right)^{n}\frac{t^n}{n!}.
\]
This completes the proof.
\end{proof}

By using (\ref{gp-1}), we get%
\begin{equation*}
e^{-\frac{1}{i\omega }t}\sum_{n=0}^{\infty }\widehat{B_{n}}(\omega )\frac{%
t^{n}}{n!}=e^{-\frac{e^{-i\omega }}{i\omega }t}.
\end{equation*}%
Integrating both sides of the above equation from $0$ to $\infty $, we get%
\begin{equation*}
\sum_{n=0}^{\infty }\widehat{B_{n}}(\omega )\frac{1}{n!}\int\limits_{0}^{%
\infty }t^{n}e^{-\frac{1}{i\omega }t}dt\int\limits_{0}^{\infty }t^{n}e^{-%
\frac{1}{i\omega }t}dt=\int\limits_{0}^{\infty }e^{-\frac{e^{-i\omega }}{%
i\omega }t}dt.
\end{equation*}%
After some calculations, with the aid of following well-known formula%
\begin{equation*}
\int\limits_{0}^{\infty }u^{n}e^{u}du=n!,
\end{equation*}%
we get the following series representation for the function $\widehat{B_{n}}(\omega )$:

\begin{theorem} (Series Representation for $\widehat{B_{n}}(\omega )$)
\begin{equation*} 
\sum_{n=0}^{\infty }\left( i\omega \right) ^{n}\widehat{B_{n}}(\omega )=\cos
\left( \omega \right) +i\sin \left( \omega \right) .
\end{equation*}
\end{theorem}

By combining equation (\ref{FourierBn}) with the following genearting
function for the Stirling numbers of the second kind:%
\begin{equation*}
\frac{1}{n!}\left( e^{z}-1\right) ^{n}=\sum\limits_{m=0}^{\infty }S_{2}(m,n)%
\frac{z^{m}}{m!}
\end{equation*}%
(\textit{cf}. \cite{Simsek2013}), we get%
\begin{equation*}
\widehat{B_{n}}(\omega )=(-1)^{n}n!\sum\limits_{m=0}^{\infty }S_{2}(m,n)%
\frac{(-i)^{m-n}w^{m-n}}{m!}.
\end{equation*}%
After some calculations, we get%
\begin{equation}
\widehat{B_{n}}(\omega )=\sum\limits_{m=0}^{\infty }\frac{i^{3m+2n}S_{2}(m+n,n)}{\binom{m+n}{n}}\frac{w^{m}}{m!}.  \label{ay2}
\end{equation}
By combining equation (\ref{FourierBn}) with the following generating
function for the Bernoulli numbers of order $-n$:%
\begin{equation*}
\left( \frac{e^{z}-1}{z}\right) ^{n}e^{zx}=\sum\limits_{m=0}^{\infty
}B_{m}^{(-n)}(x)\frac{z^{m}}{m!}
\end{equation*}%
(\textit{cf}. \cite{SrivastavaChoi2012}), we get%
\begin{equation}
\widehat{B_{n}}(\omega )=\sum\limits_{m=0}^{\infty }i^{m}B_{m}^{(-n)}(n)%
\frac{w^{m}}{m!}.  \label{ay1}
\end{equation}%
Combining (\ref{ay1}) with (\ref{ay2}) yields 
\begin{equation*}
\sum\limits_{m=0}^{\infty }\frac{i^{3m+2n}S_{2}(m+n,n)}{\binom{m+n}{n}}\frac{w^{m}}{m!}=\sum\limits_{m=0}^{\infty }i^{m}B_{m}^{(-n)}(n)%
\frac{w^{m}}{m!}.
\end{equation*}
By comparing coefficients $\frac{w^{m}}{m!}$ on the both sides of the previous equation, we arrive at the following theorem:
\begin{theorem}
Let $n$ be a positive integer and   $m$ nonnegative integer. Then we have
\begin{equation}
B_{m}^{(-n)}(n)=(-1)^{m+n}\frac{S_{2}(m+n,n)}{\binom{m+n}{n}}.  \label{ay3}
\end{equation}
\end{theorem}
Here we note that there are other different proof of equation  (\ref{ay3}), see, for detail \textit{cf}. \cite{GunREVISTA}, \cite{Simsekqd}, \cite{LuoSrivastava2011}.

Putting $m=n$ into (\ref{ay3}) yields the following known result:
\begin{equation}
B_{n}^{(-n)}(n)=\frac{S_{2}(2n,n)}{\binom{2n}{n}} \label{dd2}
\end{equation}%
(\textit{cf}. \cite{GunREVISTA}, \cite{Simsekqd}, \cite{LuoSrivastava2011}).

Observe that combining (\ref{dd2}) with the following definition of the Catalan numbers $$C_n = \frac{1}{n+1}\binom{2n}{n},$$ we have the following known result:
\[
B_{n}^{(-n)}(n)=\frac{S_{2}(2n,n)}{(n+1)\,C_n},
\]
 see  also \cite{GunREVISTA,LuoSrivastava2011}.

In this paper, we can use the following generating function for the generalized Array type polynomials, which have been constructed by  Simsek \cite{Simsek2013}:
\begin{equation}
\left(
\frac{\lambda b^t - a^t}{t}
\right)^k b^{xt}
=
\sum_{n=0}^{\infty}
\frac{S^{\,n+k}_{k}(x; a,b;\lambda)}{\binom{n+k}{k}}
\frac{t^n}{n!},
\label{stir}
\end{equation}
where $k\in\N$ and $a,b$ are real parameters. The above function also reduces to the generating function for the generalized Stirling numbers. 
\begin{theorem}
Let \textcolor{black}{$\alpha > 1$}. Then the following fractional Fourier-type generating function holds:
\begin{equation}
\widehat{B_\alpha}(\omega)e^{-i\omega x}=
\sum_{n=0}^{\infty}
\frac{\Gamma(\alpha+1)S^{n+\alpha}_{\alpha}(x; 1,e;1)}{\Gamma(n+\alpha+1)}(-i\omega)^n. \label{T}
\end{equation}
Equivalently, the above expression can be written as
\[
\widehat{B_\alpha}(\omega)=\sum_{n=0}^{\infty}
\sum_{j=0}^{n}(-i)^n\frac{\Gamma(\alpha+1)S^{j+\alpha}_{\alpha}(x; 1,e;1)}{\Gamma(j+\alpha+1)}
\omega^n.
\]
\end{theorem}
\begin{proof}
Setting $a=1$, $b=e$, and $\lambda=1$ into \eqref{stir} and replacing $t \to -i\omega$, we obtain
\[
\left(\frac{1 - e^{-i\omega}}{i\omega}
\right)^{\alpha}e^{-i\omega x}=\sum_{n=0}^{\infty}
\frac{S^{\,n+\alpha}_{\alpha}(x; 1,e;1)}{\binom{n+\alpha}{n}}
\frac{(-i\omega)^n}{n!}.
\]
Finally, using \eqref{FourierFractional}, we obtain
\begin{equation}
\widehat{B_\alpha}(\omega)e^{-i\omega x}=
\sum_{n=0}^{\infty}
\frac{S^{\,n+\alpha}_{\alpha}(x; 1,e;1)}{\binom{n+\alpha}{n}}
\frac{(-i\omega)^n}{n!},\label{11}
\end{equation}
where the binomial coefficient is defined by
\[
\binom{n+\alpha}{\alpha}=\binom{n+\alpha}{n}=
\frac{\Gamma(n+\alpha+1)}{\Gamma(\alpha+1)\Gamma(n+1)}.
\]
This completes the proof.
\end{proof}
Using the Fourier shift property, we have
\begin{equation}
\mathcal{F}\{B_\alpha(t - x)\}(\omega)=
e^{-i\omega x}\widehat{B_\alpha}(\omega).
\label{12}
\end{equation}
Combining \eqref{11} with \eqref{12}, we obtain
\begin{equation}
\mathcal{F}\{B_\alpha(t - x)\}(\omega)=
\sum_{n=0}^{\infty}
\frac{S^{\,n+\alpha}_{\alpha}(x; 1,e;1)}{\binom{n+\alpha}{n}}
\frac{(-i\omega)^n}{n!}.
\label{13}
\end{equation}
Let
\[
c_n :=\frac{S^{\,n+\alpha}_{\alpha}(x;1,e;1)}
{\binom{n+\alpha}{n} n!}.
\]
Applying the inverse Fourier transform to \eqref{13}, we obtain
\begin{equation}
B_\alpha(t-x)= \sum_{n=0}^{\infty}
c_n \mathcal{F}^{-1}\{(-i\omega)^n\}(t).
\label{in}
\end{equation}
Using that the derivatives of the Dirac delta distribution are given by
\[
\langle \delta^{(n)},\varphi\rangle=(-1)^n \varphi^{(n)}(0),\quad n\in\N,
\]
and the identity (in the sense of distributions)
\[
\mathcal{F}^{-1}\{(-i\omega)^n\}(t)=(-1)^n\delta^{(n)}(t),
\]
we arrive at the following result.
\begin{corollary}[Double Fourier via Delta Expansion]
	Let \textcolor{black}{$\alpha > 1$} and assume the distributional representation
	\[
	B_\alpha(x)=\sum_{n=0}^{\infty} c_n\delta^{(n)}(x),
	\qquad c_n\in\mathbb{C}.
	\]
	Then, in the sense of tempered distributions,
	\[
	\mathcal{F}^2\{B_\alpha\}(x)=2\pi\,B_\alpha(-x)
	=2\pi\sum_{n=0}^{\infty} c_n(-1)^n\,\delta^{(n)}(x).
	\]
	Equivalently,
	\[
	\mathcal{F}\{\widehat{B_\alpha}\}(x)=	2\pi\sum_{n=0}^{\infty} c_n\,(-1)^n\,\delta^{(n)}(x).
	\]
\end{corollary}
\begin{proof}
	For tempered distributions, the Fourier transform satisfies
	\[
	\mathcal{F}^2\{f\}(x)=2\pi f(-x).
	\]
	Applying this to $f=B_\alpha$ gives
	\[
	\mathcal{F}^2\{B_\alpha\}(x)=2\pi B_\alpha(-x).
	\]
	Using the assumed expansion and the identity
	\[
	\delta^{(n)}(-x)=(-1)^n \delta^{(n)}(x),
	\]
	we obtain
	\[
	B_\alpha(-x)=\sum_{n=0}^{\infty} c_n\delta^{(n)}(-x)=\sum_{n=0}^{\infty} c_n(-1)^n\,\delta^{(n)}(x).
	\]
	Multiplying by $2\pi$ yields the claim.
\end{proof}
\begin{remark}
	This result shows that the second Fourier transform acts diagonally on the delta-derivative basis, introducing only a reflection and the sign factor $(-1)^n$. Hence, the distributional structure of $B_\alpha$ is preserved under $\mathcal{F}^2$.
\end{remark}
\begin{corollary}[n-fold Fourier Transform]
	Let $f$ be a tempered distribution. Then
	\[
	\mathcal{F}^n f(x)=	(2\pi)^{\lfloor n/2\rfloor}
	\begin{cases}
	f(x), & n\equiv 0 \ (\mathrm{mod}\ 4),\\
	\widehat{f}(x), & n\equiv 1 \ (\mathrm{mod}\ 4),\\
	f(-x), & n\equiv 2 \ (\mathrm{mod}\ 4),\\
	\widehat{f}(-x), & n\equiv 3 \ (\mathrm{mod}\ 4).
	\end{cases}
	\]
\end{corollary}

\begin{corollary}[$n$-fold Fourier Transform of Fractional B-Splines]
	Let \textcolor{black}{$\alpha > 1$} and assume
	\[
	B_\alpha(x)=\sum_{m=0}^{\infty} c_m \delta^{(m)}(x)
	\]
	in the sense of \textcolor{black}{tempered} distributions. Then, for any $n\in\mathbb{N}$,
	\[
	\mathcal{F}^n\{B_\alpha\}(x)=(2\pi)^{\lfloor n/2\rfloor}	\sum_{m=0}^{\infty} c_m\sigma_{n,m}\,\delta^{(m)}(\pm x),
	\]
	where the sign and coefficients are given by
	\[
	\sigma_{n,m}=
	\begin{cases}
	1, & n\equiv 0 \ (\mathrm{mod}\ 4),\\
	(i)^m, & n\equiv 1 \ (\mathrm{mod}\ 4),\\
	(-1)^m, & n\equiv 2 \ (\mathrm{mod}\ 4),\\
	(-i)^m, & n\equiv 3 \ (\mathrm{mod}\ 4),
	\end{cases}
	\]
	and the argument is reflected for even $n$.
\end{corollary}
\begin{proof}
	We use the identities
	\[
	\mathcal{F}\{\delta^{(m)}(x)\}=(i\omega)^m,
	\quad
	\mathcal{F}^{-1}\{(i\omega)^m\}=(-1)^m \delta^{(m)}(x),
	\]
	and the general rule
	\[
	\mathcal{F}^2 f(x)=2\pi f(-x).
	\]
	Applying $\mathcal{F}^n$ to the expansion of $B_\alpha$ term-by-term and using the cyclic structure of the Fourier operator yields the stated result.
\end{proof}
\begin{remark}
	The above result shows that the action of the $n$-fold Fourier transform on fractional B-splines is completely determined by its effect on the delta-derivative basis. This reveals a cyclic spectral structure governed by phase factors $(\pm i)^m$ and reflections.
\end{remark}

\begin{theorem}
In the sense of \textcolor{black}{tempered} distributions,
\[
B_\alpha(t-x)=\sum_{n=0}^{\infty}	(-1)^n \frac{S^{n+\alpha}_{\alpha}(x;1,e;1)}
{\binom{n+\alpha}{n} n!} \delta^{(n)}(t).
\]
\end{theorem}
\begin{remark}
In this representation, the shift parameter $x$ does not appear explicitly in the Dirac distributions but is encoded in the coefficients.
\end{remark}
\begin{theorem}
Then $B_\alpha(t-x)$ admits the shifted representation
\[
B_\alpha(t-x)=\sum_{m=0}^\infty d_m(x) \delta^{(m)}(t-x),
\]
where
\[
d_m(x)=\sum_{n=0}^{m} c_n(-1)^n \frac{x^{m-n}}{(m-n)!}.
\]
\end{theorem}
\begin{proof}
For any test function $\varphi$, we use
\[
\langle \delta^{(n)}(t-x), \varphi(t)\rangle
= (-1)^n \varphi^{(n)}(x).
\]
By Taylor expansion,
\[
\varphi^{(n)}(x)=\sum_{k=0}^\infty \frac{x^k}{k!} \varphi^{(n+k)}(0),
\]
hence
\[
\delta^{(n)}(t-x)=\sum_{k=0}^\infty \frac{x^k}{k!} \delta^{(n+k)}(t).
\]
Equivalently,
\[
\delta^{(n)}(t)=\sum_{k=0}^\infty \frac{(-x)^k}{k!} \delta^{(n+k)}(t-x).
\]
Substituting this into the given expansion and collecting terms with the same order $m=n+k$ yields
\[
B_\alpha(t-x)=\sum_{m=0}^\infty \left(\sum_{n=0}^{m} c_n (-1)^n \frac{x^{m-n}}{(m-n)!}\right) \delta^{(m)}(t-x),
\]
which proves the claim.
\end{proof}
Next, we derive a fractional differential equation for fractional B-splines.
\begin{theorem}[Fractional Differential Equation for $B_\alpha$]	
The following equality holds in $\Psi'_+$:
\[
D^\alpha B_\alpha(t-x)=	\sum_{k=0}^{\infty}
(-1)^k \binom{\alpha}{k}\,\delta(t-x-k).
\]
In particular, $B_\alpha$ is a solution of the fractional distributional differential equation
\[
D^\alpha f(t)=\sum_{k=0}^{\infty} a_k\,\delta(t-k),
\qquad a_k=(-1)^k\binom{\alpha}{k}.
\]
\end{theorem}
\begin{proof}
We use the well-known identity \eqref{D}, we obtain
\[
D^\alpha\!\left[\frac{(t-k)_+^{\alpha-1}}{\Gamma(\alpha)}\right]=\delta(t-k).
\]
The B-spline admits the truncated power representation
\[
B_\alpha(t-x)=\frac{1}{\Gamma(\alpha)}
\sum_{k=0}^{\infty}
(-1)^k \binom{\alpha}{k}
(t-x-k)_+^{\alpha-1}.
\]
Applying $D^\alpha$ term-by-term (which is justified in $\Psi'_+$), we obtain
\[
D^\alpha B_\alpha(t-x)=	\sum_{k=0}^{\infty}
(-1)^k \binom{\alpha}{k} \delta(t-x-k),
\]
which proves the claim.
\end{proof}
\begin{remark}
The above result shows that B-splines arise naturally as solutions of fractional distributional differential equations with discrete distributional sources, linking spline theory with fractional calculus.
\end{remark}

\begin{theorem}
For $f\in\mathcal{S}(\mathbb{R})$, the we have
\begin{equation}
  \widehat{\nabla^\alpha f}(\omega)=
(1-e^{-i\omega})^\alpha \label{DD1}
\widehat{f}(\omega).  
\end{equation}
\end{theorem}
\begin{proof}
In order to give assertion of the theorem we need the following operator, which was given by Blu and Unser \cite{BluUnser2000}:
for $\textcolor{black}{\alpha > 1}$, by combining $\widehat{f(x-k)} = e^{-ik\omega}\widehat{f}(\omega)$ with the following the fractional forward difference operator:
\[
\nabla^\alpha f(x)=	\sum_{k=0}^{\infty}
(-1)^k	\binom{\alpha}{k} f(x-k),
\]
and using the binomial series, we get
\[
\sum_{k=0}^{\infty}	(-1)^k	\binom{\alpha}{k}
e^{-ik\omega}=	(1-e^{-i\omega})^\alpha.
\]
After some elementary calculations, we arrive at the desired result.
\end{proof}

\begin{theorem}
Let \textcolor{black}{$\alpha > 0$} and $n \in \mathbb{N}_0$. Then the generalized Stirling-type polynomials
$S^{n+\alpha}_{\alpha}(x;1,e;1)$ admit the explicit representation
\[
S^{n+\alpha}_{\alpha}(x;1,e;1)=	\frac{1}{\Gamma(\alpha+1)}
\sum_{k=0}^{\infty}	(-1)^k \binom{\alpha}{k}
(x+k)^{n+\alpha}.
\]
\end{theorem}
\begin{proof}
Starting from \eqref{FourierFractional}, we obtain
\begin{equation}
\widehat{B_\alpha}(\omega)e^{-i\omega x}=
\left(\frac{1-e^{-i\omega}}{i\omega}\right)^\alpha e^{-i\omega x},\label{1}
\end{equation}
we expand the fractional binomial term
\[
(1 - e^{-i\omega})^\alpha=\sum_{k=0}^{\infty}
(-1)^k \binom{\alpha}{k} e^{-i\omega k}
\]
and combining the above equation with \eqref{1} and (\ref{DD1}), we obtain
\[
\widehat{B_\alpha}(\omega)e^{-i\omega x}=	(i\omega)^{-\alpha}
\sum_{k=0}^{\infty}
(-1)^k \binom{\alpha}{k} e^{-i\omega(x+k)}.
\]
Since
\[
e^{-i\omega(x+k)}=\sum_{m=0}^{\infty}
\frac{(-i\omega)^m (x+k)^m}{m!},
\]
we also get
\[
\widehat{B_\alpha}(\omega)e^{-i\omega x}=
(i\omega)^{-\alpha}\sum_{k=0}^{\infty}(-1)^k \binom{\alpha}{k}
\sum_{m=0}^{\infty}
\frac{(-i\omega)^m (x+k)^m}{m!}.
\]
Rewriting this as a power series in $\omega$ yields
	\[
	=
	\sum_{n=0}^{\infty}
	\left[
	\sum_{k=0}^{\infty}
	(-1)^k \binom{\alpha}{k}
	\frac{(x+k)^{n+\alpha}}{\Gamma(n+\alpha+1)}
	\right]
	(-i\omega)^n.
	\]
	
	On the other hand, from the fractional generating function \eqref{T}, we have
	\[
	\widehat{B_\alpha}(\omega)e^{-i\omega x}
	=
	\sum_{n=0}^{\infty}
	\frac{\Gamma(\alpha+1)}{\Gamma(n+\alpha+1)}
	S^{n+\alpha}_{\alpha}(x;1,e;1)
	(-i\omega)^n.
	\]
	
	Equating coefficients of $(-i\omega)^n$ gives
	\[
	\frac{\Gamma(\alpha+1)}{\Gamma(n+\alpha+1)}
	S^{n+\alpha}_{\alpha}(x;1,e;1)
	=
	\sum_{k=0}^{\infty}
	(-1)^k \binom{\alpha}{k}
	\frac{(x+k)^{n+\alpha}}{\Gamma(n+\alpha+1)},
	\]
	which completes the proof.
\end{proof}

\begin{theorem}
Let \textcolor{black}{$\alpha > 1$}. Then the following fractional Fourier-type generating function holds:
\begin{equation}
\widehat{B_\alpha}(\omega)=
\sum_{n=0}^{\infty}
\frac{(-1)^\alpha\Gamma(\alpha+1)S^{n+\alpha}_{\alpha}(x; 1,e;1)}{\Gamma(n+\alpha+1)}(-i\omega)^n. \label{T2}
\end{equation}
\end{theorem}
\begin{proof}
Setting $a=e$, $b=1$, and $\lambda=1$ in \eqref{stir}, we obtain
\[
\left(
\frac{1 - e^t}{t}
\right)^\alpha
=
\sum_{n=0}^{\infty}
\frac{S^{n+\alpha}_{\alpha}(x; e,1;1)}{\binom{n+\alpha}{\alpha}}
\frac{t^n}{n!}.
\]
Replacing $t$ by $-i\omega$, it follows that
\[
\left(
\frac{1 - e^{-i\omega}}{i\omega}
\right)^\alpha
=
\sum_{n=0}^{\infty}
\frac{(-1)^\alpha S^{n+\alpha}_{\alpha}(x; e,1;1)}{\binom{n+\alpha}{\alpha}}
\frac{(-i\omega)^n}{n!}.
\]
\end{proof}
\begin{theorem}
Let \textcolor{black}{$\alpha > 1$}. Then the \textcolor{black}{fractional} B-spline admits the following inverse Fourier-type distributional representation:
\begin{equation}
B_\alpha(-x)=\sum_{n=0}^{\infty}\frac{(-1)^\alpha \Gamma(\alpha+1)}{\Gamma(n+\alpha+1)}
S^{n+\alpha}_{\alpha}(x; 1,e;1)
(-1)^n \delta^{(n)}(x), \label{14}
\end{equation}
where the equality holds in the sense of distributions.
\end{theorem}

\begin{proof}
Starting from the fractional Fourier-type generating representation, multiplying both sides of \eqref{T2} by $e^{-i\omega x}$, we obtain
\[
\widehat{B_\alpha}(\omega)e^{-i\omega x}=	\sum_{n=0}^{\infty}
\frac{(-1)^\alpha \Gamma(\alpha+1)}{\Gamma(n+\alpha+1)}
S^{n+\alpha}_{\alpha}(x; 1,e;1)(-i\omega)^n e^{-i\omega x},
\]
we integrate both sides over $\mathbb{R}$ with respect to $\omega$:
\begin{equation}
\int_{\mathbb{R}} \widehat{B_\alpha}(\omega)e^{-i\omega x} d\omega=	\int_{\mathbb{R}} \sum_{n=0}^{\infty}
\frac{(-1)^\alpha \Gamma(\alpha+1)}{\Gamma(n+\alpha+1)}
S^{n+\alpha}_{\alpha}(x; 1,e;1)(-i\omega)^n e^{-i\omega x} d\omega. \label{2}
\end{equation}
The left-hand side gives
\begin{equation}
\int_{\mathbb{R}} \widehat{B_\alpha}(\omega)e^{-i\omega x} d\omega= 2\pi B_\alpha(-x). \label{3}
\end{equation}
For the right-hand side, using the distributional identity
\begin{equation}
\int_{\mathbb{R}} (-i\omega)^n e^{-i\omega x} d\omega
= 2\pi (-1)^n \delta^{(n)}(x), \label{4}
\end{equation}
By combining \eqref{3} and \eqref{4} and substituting the result into \eqref{2}, we obtain
\[
2\pi B_\alpha(-x)=2\pi \sum_{n=0}^{\infty}
\frac{(-1)^\alpha \Gamma(\alpha+1)}{\Gamma(n+\alpha+1)}
S^{n+\alpha}_{\alpha}(x; 1,e;1)
(-1)^n \delta^{(n)}(x).
\]
Dividing both sides by $2\pi$ completes the proof.
\end{proof}
\begin{remark}
The above identity shows that the generalized Stirling-type polynomials 
act as coefficients in a distributional expansion of the \textcolor{black}{fractional} B-spline 
in terms of derivatives of the Dirac delta.
\end{remark}
\begin{theorem}	Let \textcolor{black}{$\alpha > 1$} and assume that the \textcolor{black}{fractional} B-spline admits the	distributional representation
\[
B_\alpha(-x)=\sum_{n=0}^{\infty} a_n(x)\,(-1)^n\delta^{(n)}(x),
\]
where $a_n \in C^\infty(\mathbb{R})$. Then, for every $\varphi \in \cS(\mathbb{R})$, the following identity holds:
\[
\int_{\mathbb{R}} B_\alpha(-x)\varphi(x)dx=
\sum_{n=0}^{\infty} a_n(0)\varphi^{(n)}(0).
\]
In particular, if
\[
a_n(x)=\frac{(-1)^\alpha \Gamma(\alpha+1)}{\Gamma(n+\alpha+1)}
S^{n+\alpha}_{\alpha}(x;1,e;1),
\]
then
\[
\int_{\mathbb{R}} B_\alpha(-x)\varphi(x)dx=
\sum_{n=0}^{\infty}	\frac{(-1)^\alpha \Gamma(\alpha+1)}{\Gamma(n+\alpha+1)}
S^{n+\alpha}_{\alpha}(0;1,e;1)
\varphi^{(n)}(0).
\]
\end{theorem}
\begin{proof}
We apply in \eqref{14} both sides to an arbitrary test function $\varphi \in \cS(\mathbb{R})$. This yields
\[
\int_{\mathbb{R}} B_\alpha(-x)\varphi(x)dx=
\left\langle\sum_{n=0}^{\infty} a_n(x)(-1)^n \delta^{(n)}(x),\,\varphi(x)\right\rangle.
\]
By linearity of distributions, we can interchange the summation and the pairing:
\begin{equation}
\int_{\mathbb{R}} B_\alpha(-x)\varphi(x)dx=\sum_{n=0}^{\infty}\left \langle a_n(x)(-1)^n \delta^{(n)}(x),\varphi(x) \right\rangle. \label{15}
\end{equation}
Using the property $f(x)\delta(x)=f(0)\delta(x)$ (and its extension to derivatives), we obtain
\[
a_n(x)\delta^{(n)}(x)=a_n(0)\delta^{(n)}(x),
\]
Substituting the above equation into \eqref{15}, we obtain
\[
=
\sum_{n=0}^{\infty}a_n(0)(-1)^n
\langle \delta^{(n)}(x),\varphi(x)\rangle.
\]
We now apply the well-known identity for the action of the $n$-th derivative of the Dirac distribution:
\[
\langle \delta^{(n)} ,\varphi \rangle = (-1)^n \varphi^{(n)}(0),
\]
Substituting the above equation into \eqref{15}, we obtain
\[
\sum_{n=0}^{\infty}
a_n(0)(-1)^n(-1)^n\varphi^{(n)}(0)=
\sum_{n=0}^{\infty}a_n(0)\varphi^{(n)}(0),
\]
which completes the proof.
\end{proof}
\begin{remark}
This identity shows that the action of the \textcolor{black}{fractional} B-spline on test functions
is completely determined by the derivatives of the test function at the origin,
i.e., the coefficients $a_n(0)$ encode the local distributional structure of
$B_\alpha$.
\end{remark}

\begin{definition}
For \textcolor{black}{$\alpha > 0$}, define the fractional spline polynomials
\begin{equation}
S_n^{(\alpha)}(x)=\frac{1}{\Gamma(\alpha+n+1)}
\sum_{k=0}^{\infty}	(-1)^k	\binom{\alpha+1}{k}
(x-k)_+^{\alpha+n},
\qquad n\ge0. \label{def}
\end{equation}
\end{definition}
These functions extend the fractional B-spline basis to a polynomial-type family.
\begin{figure}[h!]
    \centering
    \includegraphics[width=4cm,height= 3cm]{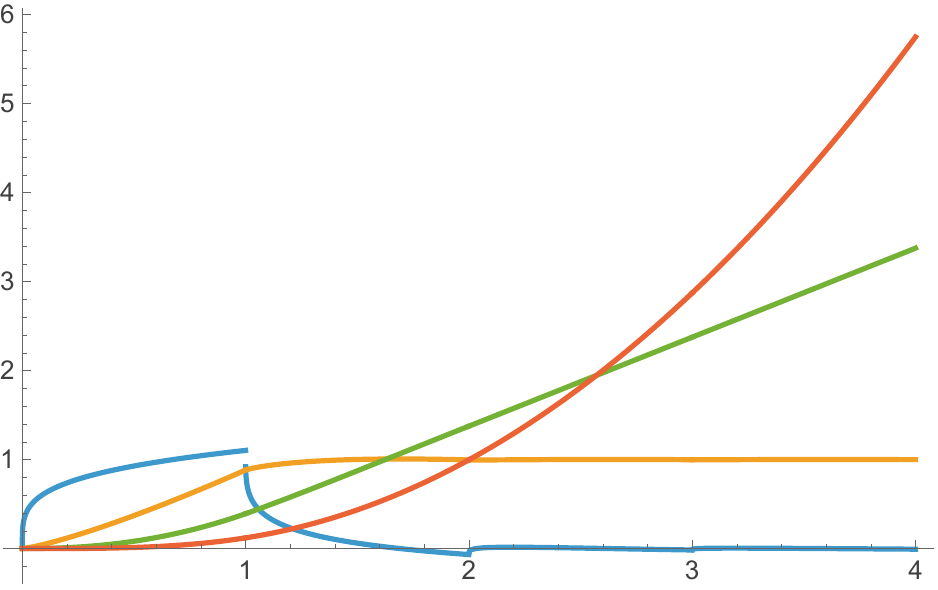}\qquad\includegraphics[width=4cm,height= 3cm]{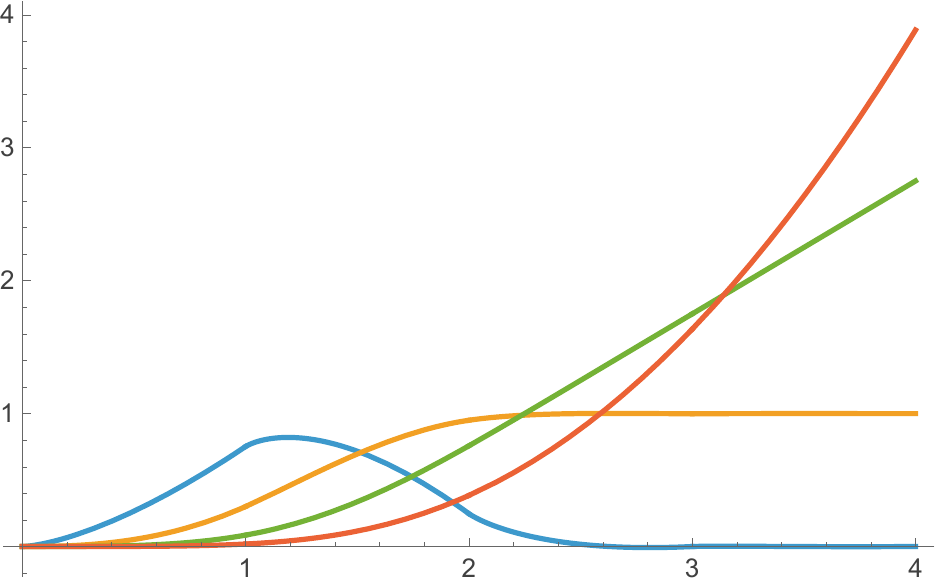}\qquad\includegraphics[width=4cm,height= 3cm]{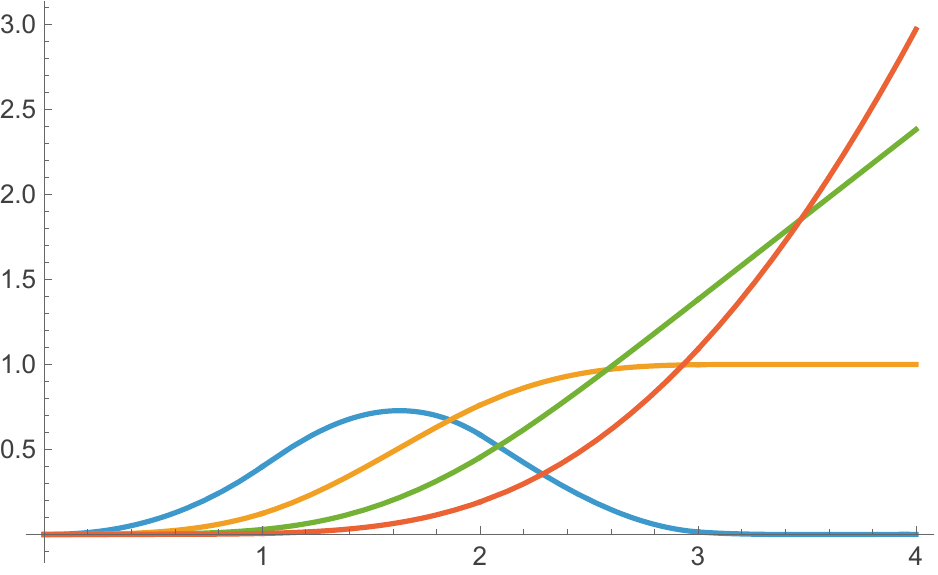}
    \caption{The first four fractional spline polynomials for $\alpha\in\{\frac14,1.5,\sqrt{5}\}$.}
    \label{Fig3}
\end{figure}
\begin{theorem} Let \textcolor{black}{$\alpha > 0$}. The fractional spline polynomials satisfy the ordinary generating function \[ 
\sum_{n=0}^{\infty} S_n^{(\alpha)}(x) t^n = \sum_{k=0}^{\infty} (-1)^k \binom{\alpha+1}{k} (x-k)_+^\alpha E_{1,\alpha+1}\!\left((x-k)_+ t\right), 
\] 
where $E_{a,b}(z)$ denotes the Mittag--Leffler function
\[ E_{a,b}(z)=\sum_{n=0}^{\infty}\frac{z^n}{\Gamma(an+b)}. 
\]
\end{theorem} 
\begin{proof}
Starting from \eqref{def}, we multiply both sides by $t^n$ and sum over $n \ge 0$, which yields the exponential generating function
\[
\sum_{n=0}^{\infty} S_n^{(\alpha)}(x)t^n = \sum_{n=0}^{\infty}t^n \frac{1}{\Gamma(\alpha+n+1)} \sum_{k=0}^{\infty} (-1)^k \binom{\alpha+1}{k} (x-k)_+^{\alpha+n}.
\] 
Changing the order of summation yields 
\[\sum_{n=0}^{\infty} S_n^{(\alpha)}(x)t^n = \sum_{k=0}^{\infty} (-1)^k \binom{\alpha+1}{k} (x-k)_+^\alpha \sum_{n=0}^{\infty} \frac{(x-k)_+^n t^n}{\Gamma(\alpha+n+1)}. 
\] 
The inner series corresponds to the Mittag--Leffler function, hence 
\[ 
\sum_{n=0}^{\infty} \frac{(x-k)_+^n t^n}{\Gamma(\alpha+n+1)} = E_{1,\alpha+1}((x-k)_+t). 
\] 
Substituting this into the previous expression completes the proof. 
\end{proof}

\end{document}